\documentclass[10pt]{article}
\usepackage{amsmath,amssymb,amsthm}
\usepackage{epsfig}
\usepackage{color}

\title{Packing parameters in graphs: New bounds and a solution to an open problem}

\author {
Doost Ali Mojdeh and Babak Samadi\thanks{Corresponding author}\\
Department of Mathematics\\
University of Mazandaran, Babolsar, Iran\\
{\tt damojdeh@umz.ac.ir}\\
{\tt samadibabak62@gmail.com$^*$}\vspace{3mm}\\
}
\date{}

\newtheorem{theorem}{Theorem}[section]

\newtheorem{lemma}[theorem]{Lemma}

\newtheorem{proposition}[theorem]{Proposition}
\newtheorem{obs}[theorem]{Observation}

\theoremstyle{definition}

\begin{document}

\maketitle

\begin{abstract}
\noindent In this paper, we investigate the packing parameters in graphs. By applying the Mantel's theorem, we give upper bounds on packing and open packing numbers of triangle-free graphs along with characterizing the graphs for which the equalities hold and exhibit sharp Nordhaus-Gaddum type inequalities for packing numbers. We also solve the open problem of characterizing all connected graphs with $\rho_{o}(G)=n-\omega(G)$ posed in [S. Hamid and S. Saravanakumar, {\em Packing parameters in graphs}, Discuss Math. Graph Theory, 35 (2015), 5--16].
\end{abstract}
{\bf Keywords:} Packing number, open packing number, Nordhaus-Gaddum inequality, open problem, triangle-free graph.\vspace{1mm}\\
{\bf AMS Subject Classifications:} 05C69.

\section{Introduction}
Throughout this paper, let $G$ be a finite graph with vertex set $V(G)$ and edge set $E(G)$. We use \cite{w} as a reference for terminology and notation which are not defined here. The {\em open neighborhood} of a vertex $v$ is denoted by $N(v)$, and the {\em closed neighborhood} of $v$ is $N[v]=N(v)\cup \{v\}$. The minimum and maximum degree of $G$ are denoted by $‎\delta(G)‎‎$ and $‎\Delta(G)‎‎$, respectively. The subset $S\subseteq V(G)$ is said to be {\em $2$-independent} if the maximum degree of the subgraph induced by it is less then two.\\
A set $S\subseteq V(G)$ is a {\em dominating set} if each vertex in $V(G)\backslash S$ has at least one neighbor in $S$. The {\em domination number} $\gamma(G)$ is the minimum cardinality of a dominating set \cite{hhs}.\\
A subset $B\subseteq V(G)$ is a {\em $2$-packing} in $G$ if for every pair of vertices $u,v\in B$, $N[u]\cap N[v]=\phi$. The {\em $2$-packing number} (or {\em packing number}) $\rho(G)$ is the maximum cardinality of a $2$-packing in $G$. The {\em open packing}, as it is defined in \cite{hs}, is a subset $B\subseteq V(G)$ for which the open neighborhoods of the vertices of $B$ are pairwise disjoint in $G$ (clearly, $B$ is an open packing if and only if $|N(v)\cap B|\leq1$, for all $v\in V(G)$. The {\em open packing number}, denoted $‎\rho‎_{o}(G)$, is the maximum cardinality among all open packings in $G$.\\
Gallant et al. \cite{gghr} introduced the concept of {\em limited packing} in graphs. They exhibited some real-world applications of it to network security, NIMBY, market saturation and codes. In fact as it is defined in \cite{gghr}, a set of vertices $B\subseteq V(G)$ is called a {\em $k$-limited packing} in $G$ provided that for all $v\in V(G)$, we have $|N[v]\cap B|\leq k$. The {\em limited packing number}, denoted $L_{k}(G)$, is the largest number of vertices in a $k$-limited packing set. It is easy to see that $L_{1}(G)=\rho(G)$. \\
In this paper, as an application of the classic theorem of Mantel \cite{m} we give upper bounds on packing and open packing numbers of triangle-free graphs and characterize the graphs obtaining equality in these bounds. In Section 3, we give lower bounds on $L_{k}(G)$, for $k=1,2$, in terms of $k$ and the diameter of $G$. Also, we prove sharp Nordhaus-Gaddum inequalities for packing numbers.\\
In \cite{hsa}, the problem of finding all connected graphs with $\rho_{o}(G)=n-\omega(G)$ was posed as an open problem. In Section 4, we exhibit a solution to this problem.


\section{Applications of Mantel's Theorem}

Our aim in this section is to establish upper bounds on $‎\rho‎_{o}(G)‎$ and $‎\rho‎(G)‎$ for a triangle-free graph $G$ in terms of its order and size. Furthermore, we characterize all triangle-free graphs attaining these bounds. We need the following well-known theorem of Mantel from extremal graph theory.
\begin{lemma}\label{L1}\emph{(\cite{m}) (Mantel's Theorem)}
If $G$ is a triangle-free graph of order $n$, then
$$|E(G)|‎\leq \lfloor n^{2}/4\rfloor$$
with equality if and only $G$ is isomorphic to $K_{\lfloor\frac{n}{2}\rfloor,\lceil\frac{n}{2}\rceil}$.
\end{lemma}
In order to characterize all triangle-free graphs attaining the upper bounds in the following theorem, we define the family $‎\Omega‎$ to consist of all triangle-free graphs $G$ that can be obtained from the disjoint union of a complete bipartite graph $K‎_{t,t}‎$ and $pK‎_{2}‎$ ($p\geq1$) by adding exactly one edge $uv$ such that $u\in V(K‎_{t,t})$ and $v\in V(pK‎_{2}‎)$, for every $u\in V(K‎_{t,t})$. Also, we define the collection $\Omega'$ by replacing $pK‎_{2}‎$ with $pK‎_{1}‎$ in the definition of $\Omega$.
\begin{theorem}
Let $G$ be a triangle-free graph of order $n$ and size $m$. If $G$ has no isolated vertex, then
$$‎\rho‎_{o}(G)‎\leq n+1-‎\sqrt{4m-2n+1}.‎‎‎‎$$
Moreover,
$$‎\rho‎(G)‎\leq n+2-‎2\sqrt{1+m}.‎‎‎‎$$
The first inequality holds with equality if and only if $G\in ‎\Omega‎‎$ and the second holds with equality if and only if $G\in \Omega'$.
\end{theorem}
\begin{proof}
Let $B$ be a maximum open packing set in $G$. Then, $|E(G[B])|‎\leq|B|/2‎$ and $|[B,V‎\setminus B‎]|‎\leq n-|B|‎$. Since $G$ is triangle-free, $|E(G[V‎\setminus ‎B])|‎\leq(n-|B|)‎‎^{2}‎/4$, by Lemma \ref{L1}. Clearly,
\begin{equation}\label{EQ7}
m=|E(G[B])|+|[B,V‎\setminus B‎]|+|E(G[V‎\setminus ‎B])|.
\end{equation}
Therefore, $m‎\leq |B|‎/2+n-|B|+(n-|B|)‎‎^{2}‎/4$ and hence
\begin{equation*}
|B|‎/2+n-|B|+(n-|B|)‎‎^{2}‎/4-m‎\geq0.‎
\end{equation*}
Solving the above inequality for $|B|$ we obtain
\begin{equation}\label{EQ8}
‎\rho‎_{o}(G)‎=|B|\leq n+1-‎\sqrt{4m-2n+1}‎‎‎‎.
\end{equation}
The upper bound on $‎\rho(G)‎$ can be proved by a similar fashion. it suffices to note that $|E(G[B])|‎=0$ when $B$ is a packing set in $G$.\\
Now, we characterize all triangle-free graphs with no isolated vertices for which the equality in the first upper bound holds. By (\ref{EQ7}) we deduce that the inequality (\ref{EQ8}) holds with equality if and only if $|E(G[B])|‎=|B|/2‎$, $|[B,V‎\setminus B‎]|‎=n-|B|‎$ and $|E(G[V‎\setminus ‎B])|‎=(n-|B|)‎‎^{2}‎/4$.\\
Let $G\in ‎\Omega‎$ and $B$ be the set of vertices of $pK‎_{2}‎$. It is easy to see that $B$ is an open packing set satisfying the three above equality. So, $B$ is a maximum open packing in $G$ and $\rho_{o}(G)=|B|=n+1-‎\sqrt{4m-2n+1}$.\\
Conversely, suppose that $G$ satisfies the equality in (\ref{EQ8}) and $B$ is a maximum open packing set in $G$. Since $|E(G[B])|‎=|B|/2‎$, $|E(G[V‎\setminus ‎B])|‎=(n-|B|)‎‎^{2}‎/4$ and $G$ is triangle-free,  $G[V‎\setminus ‎B]$ is the complete bipartite graph $K‎_{\frac{n-|B|}{2},\frac{n-|B|}{2}}‎$, by Lemma \ref{L1}. Taking into account the facts that $B$ is an open packing and $|E(G[B])|‎=|B|/2‎$, we have $G[B]=(|B|/2)K‎_{2}‎$. On the other hand, the equality $|[B,V‎\setminus B‎]|‎=n-|B|‎$ implies that every vertex in $G[V‎\setminus B‎]$ has exactly one neighbor in $V(G[B])=V((|B|/2)K‎_{2}‎)$. This shows that $G\in ‎\Omega‎$.\\
It is easy to see that the second upper bound holds with equality for $G\in \Omega'‎‎‎$.\\
Conversely, suppose that $G$ satisfies the equality in the second upper bound and $B$ is a maximum packing set in $G$. Then $B$ is independent. Moreover, $|[B,V‎\setminus B‎]|‎=n-|B|‎$ and $|E(G[V‎\setminus ‎B])|‎=(n-|B|)‎‎^{2}‎/4$. Thus, every vertex in $V‎\setminus B‎$ has exactly one neighbor in $B$. Also, $|E(G[V‎\setminus ‎B])|‎=(n-|B|)‎‎^{2}‎/4$ shows that $G[V‎\setminus ‎B]=K‎_{\frac{n-|B|}{2},\frac{n-|B|}{2}}‎$, by Lemma \ref{L1}. Therefore, $G\in \Omega'$. This completes the proof.
\end{proof}


\section{Diameter and Nordhaus-Gaddum inequalities for packing number}

Many results in domination theory have relationship with the diameter of graphs (see \cite{hhs}). In this section we exhibit tight bounds on $L_{k}(G)$ ($k=1,2$) and the sum and product of the packing number $G$ and  $\overline G$ involving the diameter. The following well-known lower bound on the domination number for a connected graph $G$ was given in \cite{hhs}:
\begin{equation}\label{EQ11}
\gamma(G)\geq\lceil\frac{diam(G)+1}{3}\rceil.
\end{equation}
In the next result we bound the $k$-limited packing numbers, $k\in\{1,2\}$, of a general connected graph $G$ from below just in terms of $k$ and its diameter. Since $\rho(G)\leq \gamma(G)$ (see \cite{gghr}), it improves the lower bound (\ref{EQ11}) for the case $k=1$.
\begin{proposition}
For any connected graph $G$ and integer $k\in\{1,2\}$,
$$\lceil\frac{k\ diam(G)+k}{3}\rceil\leq L_{k}(G).$$
\end{proposition}
\begin{proof}
Let $P$ be a diametral path in $G$ with the set of vertices $V(P)=\{v_{1},...,v_{diam(G)+1}\}$.\\
For $k=1$, $V_{1}(P)=\{v_{1},...,v_{3i+1},...,v_{3\lfloor\frac{diam(G)}{3}\rfloor+1}\}$ is a packing in $G$. For otherwise, there exists a vertex $v$ adjacent to at least two vertices in $V_{1}(G)$. This yields to a path between $v_{1}$ and $v_{diam(G)+1}$ by $v$ with length less than $diam(G)$, a contradiction. So, $\rho(G)\geq|V_{1}(P)|=\lceil\frac{diam(G)+1}{3}\rceil$.\\
For $k=2$, $V_{2}(P)=V(P)\setminus\{v_{3},...,v_{3\lfloor\frac{diam(G)+1}{3}\rfloor}\}$ is a $2$-limited packing in $G$, by a similar fashion. Therefore, $L_{2}(G)\geq|V_{2}(P)|=\lceil\frac{2\ diam(G)+2}{3}\rceil$.
\end{proof}
Nordhaus and Gaddum in 1956, gave lower and upper bounds on the sum and product of the chromatic number, in terms of the order. Since then, bounds on $‎\psi(G)+‎\psi(\overline{G})‎‎$ or $\psi(G)\psi(\overline{G})‎‎$ are called Nordhaus-Gaddum inequalities, where $‎\psi‎$ is a graph parameter. For more information about this subject the reader can consult \cite{ah}.\\
The Nordhaus-Gaddum inequalities for limited packing parameters was initiated by exhibiting the sharp upper bound $L‎_{2}(G)+L‎_{2}(\overline{G})‎‎\leq n+2‎$, for $k=2$, in \cite{s}.\\
We conclude this section by establishing upper bounds on the sum and product of the packing number ($k=1$). We first need the following useful observation.
\begin{obs}\label{Ob1}
For any graph $G$, $\rho(G)=1$ if and only if $diam(G)\leq2$.
\end{obs}
Clearly $\rho(G)+\rho(\overline G)=2$ ($\rho(G)\rho(\overline G)=1$) if and only if $diam(G)$ and $diam(\overline G)\leq2$, by Observation \ref{Ob1}. Thus, we restrict our attention to the case max$\{diam(G),diam(\overline G)\}\geq3$.
\begin{theorem}
Let $G$ and $\overline{G}$ be both connected with $\Delta'=$min$\{\Delta(G),\Delta(\overline{G})\}$ and $M=$max$\{diam(G),diam(\overline{G})\}\geq3$. Then,
$$\rho(G)+\rho(\overline G)=\rho(G)\rho(\overline G)=4\  \mbox{if}\ \ diam(G)=diam(\overline{G})=3,$$
If $diam(G)\neq diam(\overline G)$, then
$$\rho(G)+\rho(\overline G)\leq n-\lceil\frac{2M+3\Delta'-11}{3}\rceil\ and\ \rho(G)\rho(\overline G)\leq n-\lceil\frac{2M+3\Delta'-8}{3}\rceil.$$
Furthermore, these bounds are sharp.
\end{theorem}
\begin{proof}
Let $diam(G)=diam(\overline{G})=3$ and $u$ and $v$ be the end vertices of a diametral path of length $3$. It is easy to see that $\{u,v\}$ is a dominating set in $\overline G$. Therefore, $\gamma(\overline G)\leq2$. On the other hand, $\rho(\overline G)\leq \gamma(\overline G)$. Now Observation \ref{Ob1} implies that $\rho(\overline G)=2$. A similar argument shows that $\rho(G)=2$.\\
Now let $diam(G)\neq diam(\overline{G})$. Without loss of generality we may assume that $diam(G)\geq diam(\overline G)$. Since $diam(G)\geq3$ implies $diam(\overline G)\leq3$ (see \cite{w}), we have $diam(G)\geq3$ and $diam(\overline G)\leq2$. Thus, $\rho(\overline G)=1$. Now let $B$ be a maximum packing in $G$ and $u$ be a vertex of the maximum degree. Then, at most one of the vertices in $N[u]$ belongs to $B$. Let $x$ and $y$ be the end vertices of a diametral path $P$ of the length $\ell(P)=diam(G)\geq3$ in $G$. Since $diam(G[N[u]])\leq2$, at least one of the end vertices, say $x$, is in $G\setminus N[u]$ and at most three vertices of $P$ are in $N[u]$. Then $H=P\setminus N[u]$ is disjoint union of two subpaths $P_{x}$ and $P_{y}$ of $P$ beginning at $x$ and $y$, respectively, if $y\notin N[u]$ and $H=P_{x}$ if $y\in N[u]$. Moreover, $|V(H)|=|V(P_{x})|+|V(P_{y})|\geq diam(G)-2$. Since $\rho(P_{m})=\lceil\frac{m}{3}\rceil$ (see \cite{gghr}), at most $\lceil|V(P_{x})|/3\rceil+\lceil|V(P_{y})|/3\rceil$ vetrices of $H$ belong to $B$ and therefore at least $\lfloor2|V(P_{x})|/3\rfloor+\lfloor2|V(P_{y})|/3\rfloor$ vertices of $H$ belong to $V(G)\setminus B$. Thus,
\begin{equation*}
\begin{array}{lcl}
|V(G)\setminus B|&\geq& \Delta(G)+\lfloor2|V(P_{x})|/3\rfloor+\lfloor2|V(P_{y})|/3\rfloor\\
&=& \Delta(G)+\lceil(2|V(P_{x})|-2)/3\rceil+\lceil(2|V(P_{y})|-2)/3\rceil\\
&\geq& \Delta(G)+\lceil2(|V(P_{x})|+|V(P_{y})|-2)/3\rceil\\
&\geq& \Delta(G)+\lceil(2diam(G)-8)/3\rceil.
\end{array}
\end{equation*}
So, $\rho(G)=|B|\leq n-\lceil\frac{2diam(G)+3\Delta(G)-8}{3}\rceil$. This implies the upper bounds.\\
That these bounds are sharp, may be seen as follows. Let $G$ be a graph obtained from the star $K_{1,t}$, $t\geq3$, with the central vertex $u$ by adding new edges among the pendant vertices of $K_{1,t}$ provided that there exist two nonadjacent vertices $u_{1}$ and $u_{2}$ in $N(u)$ and a vetrex $w\in N(u)$ which is adjacent neither $u_{1}$ nor $u_{2}$. We add two vertices $x$ and $y$ ($x\notin N[u]$ and $y\neq u$) and consider two paths $P_{x}$ and $P_{y}$ as above, with $\ell(P_{x})\geq\ell(P_{y})$ and $\ell(P_{x})\equiv0$ (mod 3), for which the other end vertices of them are adjacent to $u_{1}$ and $u_{2}$ (if $\ell(P_{y})\geq1$), respectively.  Show this graph by $H$. Then $\Delta(G)=\Delta(H)$, $d(x,y)=diam(H)$ and the three vertices $u$, $u_{1}$ and $u_{2}$ of the diametral path belong to $N[u]$. It is easy to see that the maximum packing $B$ of $H$ contains one vertex of $N[u]$, say $w$, and $\lceil|V(P_{x})|/3\rceil+\lceil|V(P_{y})|/3\rceil$ vertices of $V(P_{x})\cup V(P_{y})$. So,
\begin{equation}\label{EQ12}
|V(H)\setminus B|=\Delta(H)+\lceil(2|V(P_{x})|-2)/3\rceil+\lceil(2|V(P_{y})|-2)/3\rceil.
\end{equation}
Moreover, since $\ell(P_{x})\equiv0$ (mod 3) and three vertices of the $x,y$-path belong to $N[u]$, we have
\begin{equation*}
\begin{array}{lcl}
|V(H)\setminus B|&=&\Delta(H)+\lceil2(|V(P_{x})|+|V(P_{y})|-2)/3\rceil\\
&=&\Delta(H)+\lceil(2diam(H)-8)/3\rceil,
\end{array}
\end{equation*}
by (\ref{EQ12}). Hence, $|B|=n-\lceil\frac{2diam(H)+3\Delta(H)-8}{3}\rceil$. Taking into account this, the sharpness of the upper bounds follows from $\rho(\overline H)=1$.
\end{proof}


\section{Characterization of graphs with $\rho_{o}(G)=n-\omega(G)$}

Hamid and Saravanakumar \cite{hsa} posed the following open problem:\vspace{1mm}\\
{\em Characterize the connected graphs of order $n\geq3$ for which $\rho_{o}(G)=n-\omega(G)$,
where $\omega(G)$ denotes the clique number of $G$.}\vspace{1mm}\\
We conclude the paper by exhibiting a solution to this problem. For this purpose, we let $\Pi_{1}$ to be $\{P_{4},P_{5},P_{6},C_{4}\}$ for $\omega(G)=2$, and for $\omega(G)\geq3$ we define $\Pi_{2}$ to be the union of all families of connected graphs described as follows (the figures (a)-(j) depict examples of graphs in the families (a)-(j)). In each case, we let $S$ to be a maximum clique.\vspace{1mm}\\
(a)\ All graphs $G$ with $\delta(G)=1$ and $\Delta(G)=n-1=\omega(G)+1$;\\
(b)\ all graphs $G$ for which the subgraph induced by $V(G)\setminus S$ is $2$-independent and each vertex in $S$ has at most one neighbor in $V(G)\setminus S$;\\
In remaining cases each vertex in $S$ has at most one neighbor in $V(G)\setminus S$.\\
(c)\ All graphs $G$ formed from adding a new vertex $y$ and joining it to at least two vertices in $S$;\\
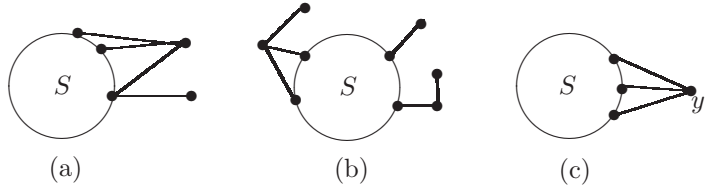
\begin{picture}(69.518,58.518)(0,0)
\put(48,26){\circle{38}}
\put(67,22){\circle*{4}}
\put(97,22){\circle*{4}}
\put(95,42){\circle*{4}}
\put(63,40){\circle*{4}}
\put(54,46){\circle*{4}}

\multiput(67,22)(.045,0){670}{\line(2,0){.9}}
\multiput(67,22)(.042,.032){670}{\line(2,0){.9}}
\multiput(65,40)(.042,.004){670}{\line(2,0){.9}}
\multiput(54,46)(.06,-.007){670}{\line(2,0){.9}}

\put(45,23){$S$}
\put(43,-8){(a)}
\end{picture}\\
\begin{picture}(69.518,58.518)(0,0)
\put(156,85){\circle{38}}
\put(175,78){\circle*{4}}
\put(190,78){\circle*{4}}
\put(190,90){\circle*{4}}
\put(172,97){\circle*{4}}
\put(184,109){\circle*{4}}
\put(140,97){\circle*{4}}
\put(124,101){\circle*{4}}
\put(140,115){\circle*{4}}
\put(136.5,80){\circle*{4}}

\multiput(175,78)(.02,0){670}{\line(2,0){.9}}
\multiput(190,78)(-.001,.02){670}{\line(2,0){.9}}
\multiput(172,97)(.018,.02){670}{\line(2,0){.9}}
\multiput(140,97)(-.023,.005){670}{\line(2,0){.9}}
\multiput(124,101)(.02,.02){670}{\line(2,0){.9}}
\multiput(124,101)(.016,-.029){670}{\line(2,0){.9}}

\put(153,82){$S$}
\put(151,51){(b)}
\end{picture}\\
\begin{picture}(69.518,38.518)(0,0)
\put(240,125){\circle{38}}
\put(257,114){\circle*{4}}
\put(257,135){\circle*{4}}
\put(260,124){\circle*{4}}
\put(286,123){\circle*{4}}

\multiput(286,123)(-.04,.003){670}{\line(2,0){.9}}
\multiput(286,123)(-.045,.02){670}{\line(2,0){.9}}
\multiput(286,123)(-.045,-.014){670}{\line(2,0){.9}}

\put(236,122){$S$}
\put(286,116.5){$y$}
\put(236,90){(c)}
\end{picture}\vspace{-28mm}\\
(d)\ all graphs $G$ formed from adding two new vertices $y$ and $z$ with $N(y)\subseteq S\setminus\{x\}$ and $N(z)=\{x\}$;\\
(e)\ all graphs $G$ obtained by adding two new vertices $y$ and $z$ for which $xy\notin E(G)$, $yz\in E(G)$ and $N(z)\cap S=\phi$;\\
(f)\ the family of graphs $G\cup\{xz\}$ in which $G$ is a graph described in (e);\\
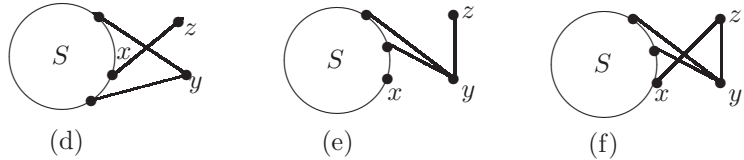
\begin{picture}(69.518,58.518)(0,0)
\put(47,29){\circle{38}}
\put(66,22){\circle*{4}}
\put(94,22){\circle*{4}}
\put(91,42){\circle*{4}}
\put(58,12){\circle*{4}}
\put(61,44){\circle*{4}}

\multiput(66,22)(.039,.032){670}{\line(2,0){.9}}
\multiput(94,22)(-.053,-.014){670}{\line(2,0){.9}}
\multiput(94,22)(-.053,.035){670}{\line(2,0){.9}}

\put(43,26){$S$}
\put(67.5,28){$x$}
\put(95,17){$y$}
\put(93,37){$z$}
\put(42,-6){(d)}
\end{picture}\\
\begin{picture}(69.518,58.518)(0,0)
\put(151,87){\circle{38}}
\put(170,80){\circle*{4}}
\put(195,80){\circle*{4}}
\put(170,92){\circle*{4}}
\put(162,104){\circle*{4}}
\put(195,104){\circle*{4}}

\multiput(195,80)(-.05,.038){670}{\line(2,0){.9}}
\multiput(195,80)(-.036,.019){670}{\line(2,0){.9}}
\multiput(195,104)(0,-.033){670}{\line(2,0){.9}}

\put(147,84){$S$}
\put(170,72){$x$}
\put(198,73){$y$}
\put(198,103){$z$}
\put(145,53){(e)}
\end{picture}\\
\begin{picture}(69.518,58.518)(0,0)
\put(252,145){\circle{38}}
\put(272,138){\circle*{4}}
\put(296,138){\circle*{4}}
\put(271,150){\circle*{4}}
\put(263,162){\circle*{4}}
\put(296,162){\circle*{4}}

\multiput(296,138)(-.05,.038){670}{\line(2,0){.9}}
\multiput(296,138)(-.036,.019){670}{\line(2,0){.9}}
\multiput(296,162)(0,-.033){670}{\line(2,0){.9}}
\multiput(271,138)(.037,.037){670}{\line(2,0){.9}}

\put(248,142){$S$}
\put(271,131){$x$}
\put(299,131){$y$}
\put(299,161){$z$}
\put(246,112){(f)}
\end{picture}\vspace{-37mm}\\
(g)\ the family of graphs formed from adding a new vertex $t$ to a graph described in (e) and joining it to $x$;\\
(h) all graphs $G$ obtained by adding new vertices $y,z$ and $w$ for which $xy\notin E(G)$, $yz\in E(G)$ and $N(\{z,w\})\cap S=\phi$;\\
(i)\ the family of graphs formed from adding a new vertex $t$ to a graph described in (h) and joining it to $x$;\\
(j)\ all graphs obtained by adding two new vertices $y$ and $z$ for which $yz,xz\in E(G)$ and $N(y)\cap S=\phi$.\\
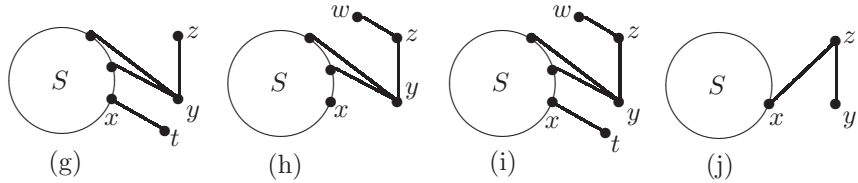
\begin{picture}(69.518,58.518)(0,0)
\put(23,27){\circle{38}}
\put(42,20){\circle*{4}}
\put(67,20){\circle*{4}}
\put(42,32){\circle*{4}}
\put(34,44){\circle*{4}}
\put(67,44){\circle*{4}}
\put(62,8){\circle*{4}}

\multiput(67,20)(-.05,.038){670}{\line(2,0){.9}}
\multiput(67,20)(-.036,.019){670}{\line(2,0){.9}}
\multiput(67,44)(0,-.033){670}{\line(2,0){.9}}
\multiput(62,8)(-.03,.017){670}{\line(2,0){.9}}

\put(19,24){$S$}
\put(39,10){$x$}
\put(70,13){$y$}
\put(70,43){$z$}
\put(64,2){$t$}
\put(18,-7){(g)}
\end{picture}
\begin{picture}(69.518,58.518)(0,0)
\put(33,26){\circle{38}}
\put(52,19){\circle*{4}}
\put(77,19){\circle*{4}}
\put(52,31){\circle*{4}}
\put(44,43){\circle*{4}}
\put(77,43){\circle*{4}}
\put(62,51){\circle*{4}}

\multiput(77,19)(-.05,.038){670}{\line(2,0){.9}}
\multiput(77,19)(-.036,.019){670}{\line(2,0){.9}}
\multiput(77,43)(0,-.033){670}{\line(2,0){.9}}
\multiput(77,43)(-.02,.012){670}{\line(2,0){.9}}

\put(29,23){$S$}
\put(53,13){$x$}
\put(80,22){$y$}
\put(80,42){$z$}
\put(52,51){$w$}
\put(28,-8){(h)}
\end{picture}
\begin{picture}(69.518,58.518)(0,0)
\put(44,26){\circle{38}}
\put(63,19){\circle*{4}}
\put(88,19){\circle*{4}}
\put(63,31){\circle*{4}}
\put(55,43){\circle*{4}}
\put(88,43){\circle*{4}}
\put(73,51){\circle*{4}}
\put(83,7){\circle*{4}}

\multiput(88,19)(-.05,.038){670}{\line(2,0){.9}}
\multiput(88,19)(-.036,.019){670}{\line(2,0){.9}}
\multiput(88,43)(0,-.033){670}{\line(2,0){.9}}
\multiput(88,43)(-.02,.012){670}{\line(2,0){.9}}
\multiput(83,7)(-.03,.017){670}{\line(2,0){.9}}

\put(40,23){$S$}
\put(60,9){$x$}
\put(91,12){$y$}
\put(91,43){$z$}
\put(63,51){$w$}
\put(85,0){$t$}
\put(39,-7){(i)}
\end{picture}
\begin{picture}(69.518,58.518)(0,0)
\put(53,25){\circle{38}}
\put(72,18){\circle*{4}}
\put(97,18){\circle*{4}}
\put(97,42){\circle*{4}}

\multiput(97,18)(0,.038){670}{\line(2,0){.9}}
\multiput(97,42)(-.038,-.036){670}{\line(2,0){.9}}

\put(49,22){$S$}
\put(72,10){$x$}
\put(100,11){$y$}
\put(100,41){$z$}
\put(48,-8){(j)}
\end{picture}\vspace{4mm}\\
We first need the following useful lemma.
\begin{lemma}\label{Lem}
Let $G$ be a connected graph of order $n$. Then, $\rho_{o}(G)\leq n-\Delta(G)+1$. Moreover, the equality holds if and only if $\Delta(G)=n-1$ and $\delta(G)=1$.
\end{lemma}
\begin{proof}
Let $B$ be an open packing in the connected graph $G$ of maximum size and $u$ be a vertex of the maximum degree $‎\Delta(G)‎$. Then at most two vertices in $N[u]$ belong two $B$ and one of them must be $u$, necessarily. Thus,
\begin{equation}\label{EQ13}
\rho‎_{o}‎(G)=|B|‎\leq n-‎\Delta(G)‎‎‎+1.
\end{equation}
We now show that the equality in (\ref{EQ13}) holds if and only if $\Delta(G)=n-1$ and $\delta(G)=1$. Let the equality holds for the graph $G$. If $u$ is a vertex of the maximum degree $‎\Delta(G)‎$, then there exists two vertices in $N[u]‎\cap B‎$, otherwise $‎\rho‎_{o}(G)‎\leq n-‎\Delta(G)$ and this is a contradiction. On the other hand, by the definition of the open packing one of these two vertices is $u$ and the other must be a pendant vertex adjacent to $u$, necessarily. Moreover, $V(G)‎\setminus N[u]‎‎\subseteq B‎$. Now the connectedness of $G$ shows that $\Delta(G)=n-1$ and $\delta(G)=1$. Conversely, let $\Delta(G)=n-1$ and $\delta(G)=1$. Then every maximum packing in $G$ contains the vertex of the maximum degree and a pendant vertex. So, $\rho‎_{o}‎(G)=n-‎\Delta(G)‎‎‎+1$.
\end{proof}
We are now in a position to exhibit the solution to the problem.
\begin{theorem}
Let $G$ be a connected graph of order $n\geq3$. Then, $\rho_{o}(G)=n-\omega(G)$ if and only if $G\in \Pi_{1}$ for $\omega(G)=2$, and $G\in \Pi_{2}$ for $\omega(G)\geq3$.
\end{theorem}
\begin{proof}
We distinguish two cases depending on the value of $\omega(G)$.\\
{\bf Case 1.} Let $\omega(G)=2$. It is a routine matter to see that $\rho_{o}(G)=n-2$ if and only if $G\in \Pi_{1}=\{P_{4},P_{5},P_{6},C_{4}\}$.\\
{\bf Case 2.} Let $\omega(G)\geq3$. Suppose that the equality holds. Clearly, $\omega(G)-1\leq \Delta(G)$. On the other hand, if $\Delta(G)\geq \omega(G)+2$ then $\rho_{o}(G)\leq n-\omega(G)-1$, by Lemma \ref{Lem}, and this is a contradiction. Therefore, $\Delta(G)\in\{\omega(G)-1,\omega(G),\omega(G)+1\}$.\\
If $\Delta(G)=\omega(G)-1$, then $G$ is a complete graph and hence $\rho_{o}(G)\neq n-\omega(G)$. Therefore, $\Delta(G)=\omega(G)$ or $\omega(G)+1$.\\
Let $\Delta(G)=\omega(G)+1$. Then, $\rho_{o}(G)\leq n-\omega(G)$ with equality if and only if $\Delta(G)=n-1$ and $\delta(G)=1$, by Lemma \ref{Lem}. In this case, $G$ belongs to the family (a).\\
Let $\Delta(G)=\omega(G)$. Let $S$ be a maximum clique in $G$. Since $S$ is a clique and $\rho_{o}(G)=n-|S|$, we deal with two possible subcases:\\
{\bf Subcase 2.1.} Let $V(G)\setminus B=S$. Then of the vertices in $V(G)\setminus S$ belong to $B$. Therefore the set $V(G)\setminus S$ is $2$-independent. Moreover, each vertex in $S$ is adjacent to at most one vertex in $V(G)\setminus S$. In this subcase, $G$ belongs to the family (b).\\
{\bf Subcase 2.2.} Let $V(G)\setminus B=(S\setminus\{x\})\cup\{y\}$, for some vertices $x\in S$ and $y\in V(G)\setminus S$. Since $V(G)\setminus(S\cup\{y\})\subseteq B$, then the vertices of the subgraph $H$ induced by this set is $2$-independent. Therefore, the components of $H$ are isolated vertices or copies of $P_{2}$. Moreover, $H$ has at most two components. For otherwise, if there exist $k\geq3$ components of $H$, then $x$ has a neighbor in at least $k-1$ components. Thus, $|N(x)\cap B|\geq k-1\geq2$, a contradiction. If $y$ has no neighbor in $V(H)$, then $y$ has all of its neighbors in $S$ and $H$ is an isolated vertex adjacent to $x$ or empty, necessarily, and we deal with all graphs in the families (c) or (d). If $y$ is adjacent to a vertex in a component $F$ of $H$, then either $y$ has a neighbor in $S\setminus\{x\}$ or $F$ is an isolated vertex adjacent to $x$ (which in this case $H$ has only one component). Considering the possible cases we can see that $G$ belongs to one of the families (e)--(j).\\
The above argument implies that $G\in \Pi_{2}$. On the other hand, it is easy to see that the equality holds for all graphs in $\Pi_{2}$. This completes the proof.
\end{proof}



\begin{thebibliography}{}

\bibitem {ah} M. Aouchiche and P. Hansen, {\em A survey of Nordhaus-Gaddum type relations}, Discrete Appl. Math. {\bf 161} (2013), 466--546.
\bibitem{gghr} R. Gallant, G. Gunther, B.L. Hartnell and D.F. Rall, {\em Limited packing in graphs}, Discrete Appl. Math. {\bf 158} (2010), 1357--1364.
\bibitem{hsa} S. Hamid and S. Saravanakumar, {\em Packing parameters in graphs}, Discuss
Math. Graph Theory, {\bf 35} (2015), 5--16.
\bibitem{hhs} T.W. Haynes, S.T. Hedetniemi and P.J. Slater, Fundamentals of Domination in Graphs, New York, Marcel Dekker, 1998.
\bibitem {hs}  M. Henning and P.J. Slater, {\em Open packing in Graphs}, J. Combin. Math. Combin. Compu., {\bf 28} (1999), 5--18.
\bibitem{m} W. Mantel. Problem 28. In Wiskundige Opgaven, {\bf 10} (1907), 60--61.
\bibitem {s} B. Samadi, {\em On the $k$-limited packing numbers in graphs}, Discrete Optim. {\bf 22} (2016), 270--276.
\bibitem{w} D.B. West, Introduction to Graph Theory (Second Edition), Prentice Hall, USA, 2001.

\end{thebibliography}
\end{document}